\documentclass[dvipdfmx, 11pt,a4paper]{article}
\usepackage[all]{xy}
\usepackage[utf8]{inputenc}
\usepackage{amsmath,bm}
\usepackage{amssymb}
\usepackage{amsthm}
\usepackage{mathtools}
\usepackage{cases}
\usepackage{float}
\usepackage{here}
\usepackage{comment}
\usepackage{ulem}
\usepackage[T1]{fontenc}
\usepackage{textcomp}
\usepackage{graphicx, color}
\usepackage{tikz,epic,eso-pic}

\theoremstyle{plain}
\newtheorem{thm}{Theorem}[section]
\newtheorem{lem}[thm]{Lemma}
\newtheorem{cor}[thm]{Corollary}
\newtheorem{prop}[thm]{Proposition}

\theoremstyle{definition}
\newtheorem{df}[thm]{Definition}

\theoremstyle{remark}
\newtheorem{rem}[thm]{Remark}

\numberwithin{equation}{section}
\allowdisplaybreaks[3]

\newcommand{\R}{\mathbb{R}}

\DeclareMathOperator{\rank}{rank}
\usepackage{xcolor}

\title{
Geometric interpretation of magnitude
}

\author{Yasuhiko Asao and Kiyonori Gomi}

\date{\empty}
\begin{document}

\maketitle
\begin{abstract}
    For an $n\times n$ positive definite symmetric matrix $Z$ with $Z_{ii} = 1$ for all $i$, we show that there exists a set of vectors $V_Z\subset \R^n$ such that the radius $R$ of the circumsphere of $V_Z$ satisfies ${\rm Mag}\ Z = (1-R^2)^{-1}$. This leads us to interpret geometrically several known and new facts on magnitude. In particular, we show that ${\rm Mag}\ Z_{X}< n$ for an $n$-point metric space $X$ of negative type with $n>1$. This result gives a negative answer to a problem given by Gomi--Meckes \cite{GM}. Furthermore, we also have a similar geometric description of magnitude for general real symmetric matrix $Z$ with $Z_{ii} = 1$ for all $i$. In this case, the radius corresponds to that of a circum-quasi-sphere, namely the set of points having a prescribed norm in a vector space endowed with an indefinite inner product.
\end{abstract}

\tableofcontents 

\section{Introduction} \label{sec:introduction}
{\it Magnitude} is a numerical invariant of finite metric spaces defined by Leinster \cite{L}. It is defined via an operation applied to a general square matrix. Namely, for a non-degenerate square matrix $Z$, we define 
\[
{\rm Mag}\ Z := \sum_{ij}(Z^{-1})_{ij},
\]
and for a finite metric space $(X = \{x_1, \dots, x_n\}, d)$, its magnitude is defined as
\[
{\rm Mag}\ X := {\rm Mag}\ Z_X = {\rm Mag}\ (e^{-d(x_i, x_j)})_{ij}.
\]
When the matrix $Z$ is symmetric, we can define its magnitude, even in the degenerate case, as
\[
{\rm Mag}\ Z = \sum_{i}w_i,
\]
provided that we have a vector $w$, called {\it a magnitude weighting}, satisfying $Zw = (1, \dots, 1)^t$, where we denote the transpose of a matrix $C$ by $C^t$. Note that the matrix $Z_X$ for a finite metric space $(X, d)$ is a real symmetric matrix with $Z_{ii} = 1$ for all $i$. The main result in this paper is the following. 

 \begin{thm}\label{thm1}
 Let $Z \in M_n(\R)$ be a positive definite symmetric matrix with $Z_{ii} = 1$ for all $i$. We can choose a non-degenerate square matrix $V = (v_1\ \dots\ v_n)$ satisfying $Z=V^t\cdot V$. Let $R$ be the radius of the circumsphere of points $\{v_1, \dots, v_n\} \subset \R^n$. Namely, $R$ is the radius of the $(n-2)$-dimensional sphere appearing as the intersection of the unit sphere in $\R^n$ and the affine span 
 \[
 {\rm Aff}\{v_1, \dots, v_n\} = \{Va \mid a\in \R^n, \sum_ia_i = 1\}.
 \]
 Then we have
 \[
 {\rm Mag}\ Z = \frac{1}{1-R^2}.
 \]
 Furthermore, for $Va \in \R^n$ being the center of the circumsphere, we have the following expression of a magnitude weighting of $Z$ :
 \[
 w = \frac{1}{1-R^2}a.
 \]
 \end{thm}
 Note here that we always have $R<1$ since linearly independent vectors $v_1, \dots, v_n$ are distributed on the unit sphere by the assumption $Z_{ii}= 1$ for all $i$. Furthermore, we can drop the assumptions `positive definite' and `non-degenerate', and obtain an extended claim which is stated as Theorem \ref{main} in the main body. There we should use `quasi-sphere', a set of points with a prescribed `norm' in a pseudo-Euclidean space according to the signature of the matrix $Z$, instead of the usual sphere.

 Such geometric descriptions lead us to understand the following subjects geometrically: 
\begin{itemize}
\item Upper bound of magnitude for negative type finite metric spaces,
\item Criterion for the existence of a magnitude weighting,
\item Criterion for the existence of a positive weighting, namely a magnitude weighting $w$ with $w_i>0$,
\item Rayleigh quotient-like expression of magnitude for positive semi-definite matrices by Leinster \cite{L},
\item Relation between magnitude and {\it the spread} defined by Willerton \cite{W}.
\end{itemize}

In particular, we explain the first subject more in detail here. From the beginning of the history of magnitude theory, the magnitude of finite metric spaces had been anticipated to represent {\it the effective number of points} with respect to the scaling constant $t$. Namely, for an $n$-point  metric space $(X, d)$ and its rescaling $tX:=(X, td)$ for $t>0$, the magnitude ${\rm Mag}\ tX$ is anticipated to approach 
\[
\begin{cases}
1 & \text{as } t \to 0, \\
n & \text{as } t \to \infty, \\
\text{number of clusters (in an appropriate sense)} & \text{ for intermediate } t.
\end{cases}
\]
For example, let us consider the case that $X$ consists of two points. Then it looks like almost one point when we are very far from $X$, which corresponds to the case that $t$ is very small. It becomes more like `a two point space' as we get closer to $X$, namely as $t$ gets larger.  On the other hand, it is known that the complete bipartite graph $K_{2, 3}$ with path metric satisfies 
\[
{\rm Mag}\  tK_{2, 3} = \begin{cases} -\infty & t \to \log\sqrt{2}-0, \\ \infty & t \to \log\sqrt{2}+0,   \end{cases}
\]
which implies the above anticipation is not true for arbitrary finite metric spaces \cite{L}.  Hence the problem is to determine which class of metric spaces satisfies the anticipation, in particular, to determine {\it the upper bound} $\sup_{t>0}{\rm Mag}\ tX$. Our contribution in this context is the determination of this value for the metric spaces with a nice property, i.\ e.\   {\it finite metric spaces of negative type}. Recall that a finite metric space $(X, d)$ is {\it of negative type} if the metric space $(X, \sqrt{d})$ can be embedded into a Euclidean space. This class of metric spaces is relatively well-studied in magnitude theory and the theory of embeddings of metric spaces (\cite{L}, \cite{sch1}, \cite{sch2}, \cite{W}). It is known by Schoenberg \cite{sch2} that a finite metric space $(X, d)$ is of negative type if and only if it is {\it stably positive definite}, namely the matrix $Z_{tX}$ is positive definite for all $t>0$. Now we show the following based on the geometric description in Theorem \ref{thm1}.

\begin{thm}\label{upbd}
For an $n$-point metric space $(X, d)$ of negative type with $n > 1$, we have
\[
{\rm Mag}\ X < n.
\]
\end{thm}

We remark that the second author and Meckes \cite{GM} showed the inequality ${\rm Mag}\ X\leq n$ for $n$-point metric spaces of negative type. Their proof relies on a non-trivial matrix inequality, while the proof in this paper does not. In addition, they left a problem to determine whether or not there exists an $n$-point metric space of negative type satisfying ${\rm Mag}\ X = n > 3$. Our Theorem \ref{upbd} solves their problem negatively.

\medskip

In Theorem \ref{thm1}, the circumsphere of the points  $\{ v_1, \ldots, v_n \}$ is considered. Considering the circumsphere of the points $\{ 0, v_1, \ldots, v_n \}$ instead, we can have another geometric description of the magnitude different from that in Theorem \ref{thm1}:

\begin{thm} \label{thm:another_description}
Let $Z \in M_n(\R)$ be a positive definite symmetric matrix with $Z_{ii} = 1$ for all $i$.  Let $v_1, \dots, v_n \in \R^n$ be linearly independent vectors such that $Z = V^t\cdot V$ with $V = (v_1\ \dots\ v_n)$. Then we have
\[
{\rm Mag}\ Z = 4 R^2,
\]
where $R$ is the radius of the circumsphere of the points $\{0, v_1, \dots, v_n \}$. 
\end{thm}

This formula would also have the potential to reproduce the facts about the magnitude reviwed so far, but we will not pursue this route in this paper, except for an alternative proof of Theorem \ref{upbd}.

\medskip

Finally, we note that, during the final stage of preparing this paper, an independent work by Devriendt \cite{D} appeared, containing the same formula as in Theorem \ref{thm1}.

\medskip

The rest of this paper is organized as follows. After preliminaries for fixing notations in Section \ref{prel}, we give a proof for the geometric interpretation of the magnitude described by the radius of a sphere or quasi-sphere in Section \ref{sec:prf}. In Section \ref{sec:upbd}, we prove Theorem \ref{upbd}. In Section \ref{sec:geomin}, we discuss the geometric interpretation of the other facts on magnitude listed above. Finally, in Section \ref{sec:another}, we provide the other gometric description as in Theorem \ref{thm:another_description} and the other proof of Theorem \ref{upbd}.

\subsubsection*{Acknowledgements}
The authors are grateful to Mark Meckes for pointing out our crucial error in the first version of the manuscript. They also thank Kouichi Taira and Tomoshige Yukita for fruitful discussions and comments.  The first author was supported by JSPS KAKENHI 24K16927.

\section{Preliminaries}\label{prel}
Throughout this paper, all vector spaces are defined over the real field $\mathbb{R}$.
\subsection{Linear algebra}
In this subsection, we collect basic facts on inner product spaces for the sake of fixing terminologies. Proofs for well-known facts are omitted; readers may refer to standard textbooks if necessary.

\begin{df}
Let $V$ be a vector space.
\begin{enumerate}
\item A symmetric bilinear form $\langle -, -\rangle : V\otimes V \to \R$ is called  {\it an inner product on} $V$.
\item Let $B = \{v_i\}_i \subset V$ be a basis of $V$.  {\it The representation} of an inner product $\langle-, -\rangle$ with respect to the basis $B$ is the symmetric matrix $(\langle v_i, v_j\rangle)_{ij}$. For $p, q, r$ being the number of positive, negative and zero eigenvalues of the matrix $(\langle v_i, v_j\rangle)_{ij}$ respectively, we call the tuple $(p, q, r)$ {\it the signature} of this inner product. By  Sylvester's law of inertia,  the signature of an inner product is independent of the choice of the basis $B$.
\item For $v \in V$ and an inner product $\langle -, -\rangle$ on $V$, we denote the map $V \to \R ; w \mapsto \langle v, w\rangle$ by $\langle v, -\rangle$. An inner product $\langle -, -\rangle$ is called {\it non-degenerate} if  $\langle v, -\rangle = 0$ implies $v = 0$. Equivalently, $\langle -, -\rangle$ is non-degenerate if its signature $(p ,q, r)$ satisfies $r=0$.

\item A non-degenerate inner product $\langle -, -\rangle$ is called {\it positive definite} if $q=0$, and called {\it indefinite} otherwise. Equivalently, $\langle -, -\rangle$ is positive definite if $\langle v, v \rangle > 0$ for all $0 \neq v \in V$.
\end{enumerate}
\end{df}
In the following, we denote the signature of a non-degenerate inner product space by $(p, q)$ instead of $(p ,q, 0)$.
\begin{df}
Let $(V, \langle -, - \rangle)$ be a non-degenerate inner product space of signature $(p,q)$.
A basis $\{\varepsilon_i\}_i \subset V$ is called {\it orthonormal} if
\[
\langle \varepsilon_i, \varepsilon_j \rangle =
\begin{cases}
0 & (i \ne j),\\
1 & (1 \le i = j \le p),\\
-1 & (p+1 \le i = j \le p+q).
\end{cases}
\]
\end{df}
The fact that every real symmetric matrix can be diagonalized by an orthogonal matrix implies the following.
\begin{prop}\label{prop:orthobasis}
Every non-degenerate inner product space $(V, \langle -, - \rangle)$ admits an orthonormal basis.
\end{prop}
\begin{df}
Let $(V, \langle -, - \rangle)$ be an inner product space. A linear subspace $W \subset V$ is called {\it non-degenerate} if the restriction of $\langle -, - \rangle$ to $W$ is non-degenerate.
\end{df}
\begin{df}
    For a linear subspace $W$ of an inner product space $(V, \langle -, - \rangle)$, we define its {\it orthogonal complement} $W^\perp = \{ v \in V \mid \langle v, w \rangle = 0, \forall w \in W \}$.
\end{df}
The non-degeneracy of a linear subspace $W$ implies that $V = W \oplus W^{\perp}$, which leads to the following.
\begin{prop}\label{perp}
Let $(V, \langle -, - \rangle)$ be a non-degenerate inner product space of signature $(p, q)$. Suppose $W$ is a non-degenerate linear subspace of $V$, and the restriction of $\langle -, - \rangle$ to $W$ has the signature $(p', q')$. Then the orthogonal complement $W^\perp$ is also a non-degenerate linear subspace, and the restriction of $\langle -, - \rangle$ to $W^\perp$ has the signature $(p - p', q - q')$.
\end{prop}
Propositions \ref{prop:orthobasis} and \ref{perp} imply the following.
\begin{prop}\label{lem:extendbasis}
Let $W$ be a non-degenerate linear subspace of a non-degenerate inner product space $(V, \langle -, - \rangle)$.
Every orthonormal basis of $W$ can be extended to an orthonormal basis of $V$.
\end{prop}

\begin{df}
Let $W$ be a linear subspace of an inner product space $(V, \langle -, - \rangle)$. A vector $v \in V$ is said to be {\it perpendicular to} $W$ if $\langle v, W\rangle = 0$.
\end{df}

\begin{prop}\label{choco}
Let $(V, \langle -, - \rangle)$ be a non-degenerate inner product space. For a linear subspace $W\subsetneq V$, we have a non-zero vector $v \in  V$ that is perpendicular to $W$. 
\end{prop}
\begin{proof}
    Let $w_1, \dots, w_m \in W$ be a basis of $W$, and define a linar map $f : V \to \R^m ; v \mapsto (\langle w_1, v\rangle, \dots, \langle w_m, v\rangle)^t$. From the assumption $m = \dim W< \dim V$, we have ${\rm ker}\ f \neq 0$, which implies the claim.
\end{proof}
\begin{df}
Let  $ (V, \langle -, - \rangle)$ be a non-degenerate inner product space. A non-zero vector $v \in V$ is called 
\[
\begin{cases}
\text{{\it spacelike}} & \text{if } \langle v, v \rangle > 0,\\
\text{{\it timelike}} & \text{if } \langle v, v \rangle < 0,\\
\text{{\it lightlike}} & \text{if } \langle v, v \rangle = 0.
\end{cases}
\]
\end{df}

We have the following from Proposition \ref{perp}.
\begin{prop}\label{codim1}
Let $(V, \langle-, -\rangle)$ be an $n$-dimensional non-degenerate inner product space. For an $(n-1)$-dimensional linear subspace $W \subset V$, the following are equivalent.
\begin{enumerate}
\item $W$ is non-degenerate
\item There exists a non-lightlike vector $v \in V$ that is perpendicular to $W$.
\end{enumerate}
\end{prop}

\subsection{Pseudo-Euclidean space}
In the following, $I_{p, q, r}$ denotes the square matrix defined by 
\[
(I_{p, q, r})_{ij} = \begin{cases} 1 & 1\leq i=j \leq p, \\ -1 & p+1 \leq i=j \leq p+q,\\ 0 & \text{otherwise}, \end{cases}
\]
with appropriate size in the context. We abbreviate $I_{p, q} = I_{p, q, 0}$ and $I_p = I_{p, 0, 0}$.
\begin{df}
\begin{enumerate}
\item Let $Z \in M_n(\R)$ be a symmetric matrix. We denote the inner product on $\R^n$ defined by $x\otimes y \mapsto x^tZy$ by $\langle -, -\rangle_{Z}$. We abbreviate $\langle -, -\rangle_{p, q, r} := \langle -, -\rangle_{I_{p, q, r}}$ and $\langle -, -\rangle_{p, q} := \langle -, -\rangle_{I_{p, q}}$.
\item We define the signature of a symmetric matrix $Z \in M_n(\R)$ as the signature of the inner product space $(\R^n, \langle-, -\rangle_Z)$.
\item We call the inner product space $(\R^{p+q}, \langle-, -\rangle_{p, q})$ {\it the pseudo-Euclidean space with signature} $(p, q)$. We denote it by $\R^{p, q}$.
\end{enumerate}
\end{df}

\begin{prop}\label{standard}
For a symmetric matrix $Z \in M_n(\R)$, the inner product space $(\R^n, \langle-, -\rangle_Z)$ is isomorphic to the inner product space $(\R^n, \langle-, -\rangle_{p, q, r})$ for some $p, q, r$.
\end{prop}

\subsection{Affine space}
\begin{df}
Let $V$ be a vector space.
\begin{enumerate}
\item  {\it An affine subspace} of $V$ is a subset $A \subset V$ which is expressed as $A = v + W = \{v+w \mid w \in W\}$ for a vector $v \in V$ and a linear subspace $W \subset V$. We define {\it the dimension} of this affine subspace $A$ by $\dim A = \dim W$.
\item For a subset $S \subset V$, we denote the smallest affine subspace of $V$ containing $S$ by ${\rm Aff}\ S$.
\end{enumerate}
\end{df}
When $S$ is a finite set, we have ${\rm Aff} \ S = v + {\rm Span}\{u-v\}_{u\in S}$ and $\R v + {\rm Span}\{u-v\}_{u\in S} = {\rm Span}\ S$ for $v \in S$, where ${\rm Span}\ S$ denotes the smallest linear subspace of $V$ containing $S$. This leads to the following.
\begin{prop}\label{affine}
Let $V$ be a vector space, and let $S = \{v_1, \dots, v_n\}\subset V$ be a subset.
\begin{enumerate}
\item We have
\[
{\rm Aff}\ S = \{\sum_{i=1}^na_iv_i \mid a_i \in \R, \sum_i a_i = 1\}.
\]
\item We have $\dim {\rm Aff}\ S \geq \dim {\rm Span}\ S-1$.
\end{enumerate}
\end{prop}
\begin{df}
An affine subspace $A$ of an inner product space $(V, \langle-, -\rangle)$ is called {\it non-degenerate} if it is expressed as $A = v + W$ for a vector $v \in V$ and a non-degenerate linear subspace $W \subset V$. 
\end{df}
\begin{df}
    Let $(V, \langle-, -\rangle)$ be an inner product space, and let $A = v + W$ be an affine subspace of $V$. A vector $u \in V$ is said to be {\it perpendicular to} $A$ if $\langle u, W\rangle = 0$.
\end{df}

\subsection{Quasi-sphere}
\begin{df}
For  $0 \neq K\in \R$, the submanifold 
\[
\{v \in \R^{p, q} \mid \langle v, v \rangle_{p, q} = K\}
\]
of $\R^{p+q}$ considered as a pseudo-Riemannian manifold with the metric induced from $\langle -, - \rangle_{p, q}$ is called {\it a quasi-sphere (with its radial scalar square $K$)}, which is denoted by $S_{p, q}(K)$.
\end{df}
The following is well-known and standard.
\begin{prop}
If the quasi-sphere $S_{p, q}(K)$ for $0 \neq K \in \R$ is not empty, it has a constant sectional curvature $K^{-1}$ as a pseudo-Riemannian manifold.
\end{prop}
Proposition \ref{standard} implies the following.
\begin{prop}
Let $(\R^n, \langle -, -\rangle)$ be a non-degenerate inner product space of signature $(p, q)$. For $0 \neq K \in \R$ and $x \in \R^n$, the submanifold
\[
 \{v \in \R^n \mid \langle x-v, x-v\rangle = K\}
\]
of $\R^n$ considered as a pseudo-Riemannian manifold with the metric induced from $\langle -, - \rangle$ is isomorphic to the quasi-sphere $S_{p, q}(K)$ as a pseudo-Riemannian manifold.
\end{prop}
\begin{df}
For the quasi-sphere $\{v \in \R^n \mid \langle x-v, x-v\rangle = K\}$ defined in a non-degenerate inner product space $(\R^n, \langle -, -\rangle)$ for some $K\neq 0$, we call $x \in \R^n$ {\it the center} of this quasi-sphere.
\end{df}

\begin{prop}\label{section}
Let $A$ be a non-degenerate $(p+q-1)$-dimensional affine subspace of $\R^{p, q}$. Then $A\cap S_{p, q}(1)$ is a submanifold of $A$, and it is isomorphic to a quasi-sphere as a pseudo-Riemannian manifold with the metric induced from $A$. In precise, we have the following : Let $A = v + W$ for a non-degenerate linear subspace  $W \subset \R^{p, q}$  and a non-lightlike vector $v \in \R^{p, q}$. We can assume that $\langle v, W\rangle_{p, q} = 0$ by Proposition \ref{lem:extendbasis}. Suppose that the signature of the inner product $\langle -, -\rangle_{p, q}$ restricted to $W$ is $(p', q')$. Then we have 
\[
A\cap S_{p, q}(1) \cong S_{p', q'}(1-\langle v, v\rangle_{p, q})
\]
as pseudo-Riemannian manifolds.
\end{prop}
\begin{proof}
Let $A = v + W$ for a non-degenerate linear subspace  $W \subset \R^{p, q}$  and a non-lightlike vector $v \in \R^{p, q}$ with $\langle v, W\rangle_{p, q} = 0$. Since $\dim A = p+q-1$, we have $\R^{p+q} = \R v \oplus W$, and thus
\[
A =  \{x \in \R^{p, q} \mid  \langle v, x\rangle_{p, q} = \langle v, v\rangle_{p,q}\}.
\]
Hence we have
\begin{align*}
A\cap S_{p, q}(1) &= \{x \in \R^{p, q} \mid \langle x, x\rangle_{p, q} = 1, \langle v, x\rangle_{p, q} = \langle v, v\rangle_{p,q} \} \\
&= \{x \in A \mid \langle x-v, x-v\rangle_{p, q}  = 1 -\langle v, v\rangle_{p,q} \}.
\end{align*}
This completes the proof.
\end{proof}
Note that $A\cap S_{p, q}(1)$ is not isomorphic to $S_{p', q'}(1)$ since we have $\langle v, v\rangle_{p, q} \neq 0$.

\subsection{Magnitude}
We denote the vector $(1, \dots, 1)^t \in \R^n$ by $1_n$ in the following.
\begin{df}
Let $Z \in M_n(\R)$ be a symmetric matrix.
\begin{enumerate}
\item A vector $w \in \R^n$ is said to be {\it a magnitude weighting of } $Z$ if it satisfies $Zw = 1_n$. A magnitude weighting $w$ is {\it positive} if it satisfies $w_i>0$ for all $i$.
\item When $Z$ admits a magnitude weighting $w$, we define its {\it magnitude} by ${\rm Mag}\ Z = \sum_{i}w_i$. When $Z$ is non-degenerate, it always admits a magnitude weighting, and  it is equivalent to define ${\rm Mag}\ Z = \sum_{ij}(Z^{-1})_{ij}$.
\end{enumerate}
\end{df}
In the following, a metric space is said to be {\it finite} if it consists of finitely many points.
\begin{df}
Let $(X = \{x_1, \dots, x_n\}, d)$ be a finite metric space. 
\begin{enumerate}
\item We define a symmetric matrix $Z_X := (e^{-d(x_i, x_j)})_{ij} \in M_n(\R)$.
\item  {\it The magnitude of } $X$ is defined by ${\rm Mag}\ X = {\rm Mag}\ Z_X$ if it exists.
\item For $t> 0$, we define a metric space $tX := (X, td)$. We also define a metric space $\sqrt{X}:= (X, \sqrt{d})$.
\item $X$ is said to be {\it positive (semi-)definite} if $Z_X$ is a positive (semi-)definite matrix. Also, $X$ is said to be {\it stably positive (semi-)definite} if $Z_{tX}$ is a positive (semi-)definite matrix for all $t>0$.
\item $X$ is said to be {\it of negative type} if $\sqrt{X}$ can be embedded into a Euclidean space.
\end{enumerate}
\end{df}
\begin{prop}[\cite{M}]
For a finite metric space $X$, the following are equivalent.
\begin{enumerate}
\item $X$ is stably positive definite,
\item $X$ is stably positive semi-definite,
\item There is a sequence $\{t_i>0\}_i$ converging to $0$ as $i\to \infty$ such that $t_iX$ is positive definite for all $i$.
\item $X$ is of negative type.
\end{enumerate}
\end{prop}

\section{Main result}\label{sec:prf}
Let $Z \in M_n(\R)$ be a symmetric matrix with $Z_{ii} = 1$ for all $i$. Suppose that the signature of $Z$ is $(p, q, r)$. Since every real symmetric matrix can be diagonalized by an orthogonal matrix, there exist linearly independent vectors $v_1, \dots, v_n \in \R^n$ such that $Z = V^tI_{p, q, r}V$ with $V = (v_1\ \dots\ v_n)$. We define a projection $\pi^{p,q} : \R^n \to \R^{p+q}$ by $\pi^{p, q}(x_1, \dots, x_n)^t := (x_1, \dots, x_{p+q})^t $, and we denote $\pi^{p, q}v$ by $v^{p, q}$ for $v \in \R^n$. In this section, we give a proof for the following our main theorem.

\begin{thm}\label{main}
Let $Z \in M_n(\R)$ be a symmetric matrix with $Z_{ii} = 1$ for all $i$. Suppose that the signature of $Z$ is $(p, q, r)$ and let $v_1, \dots, v_n \in \R^n$ be linearly independent vectors such that $Z = V^tI_{p, q, r}V$ with $V = (v_1\ \dots\ v_n)$. If $Z$ admits a magnitude weighting and ${\rm Mag}\ Z \neq 0$, then we have
\[
{\rm Mag}\ Z = \frac{1}{1 - K^{-1}},
\]
where $K\neq 0$ is the sectional curvature of the quasi-sphere ${\rm Aff}\{v_1^{p, q}, \dots, v_n^{p, q}\}\cap S_{p, q}(1)$. Furthermore, for the center $V^{p, q}a \in {\rm Aff}\{v_1^{p, q}, \dots, v_n^{p, q}\}, \sum_ia_i=1$ of the quasi-sphere ${\rm Aff}\{v_1^{p, q}, \dots, v_n^{p, q}\}\cap S_{p, q}(1)$, a magnitude weighting $w$ of $Z$ is described as
\[
w = \frac{1}{1-K^{-1}}a =  \frac{1}{\langle V^{p, q}a, V^{p, q}a\rangle_{p, q}}a = \frac{1}{\langle a, a\rangle_{Z}}a,
\]
where $V^{p, q} = (v_1^{p, q}\  \dots\ v^{p, q}_n)$. In particular, the following are equivalent: 
\begin{enumerate}
\item $Z$ admits a positive weighting, 
\item $Z$ is spacelike, and the center of the quasi-sphere ${\rm Aff}\{v_1^{p, q}, \dots, v_n^{p, q}\}\cap S_{p, q}(1)$ belongs to the interior of the convex hull ${\rm Conv}\{v_1^{p, q}, \dots, v_n^{p, q}\}$ of the points $v_1^{p, q}, \dots, v_n^{p, q}$,
\end{enumerate}
where we refer to Definition \ref{causal} for the term `spacelike'.
\end{thm}
Note that the sectional curvature of a sphere of radius $R$ is $R^{-2}$, and hence Theorem \ref{main} implies Theorem \ref{thm1}.
We start from the characterization of the existence of a magnitude weighting by using $V$.
\begin{prop}\label{weighting}
The following are equivalent.
\begin{enumerate}
\item $0\not\in{\rm Aff}\{v_1^{p, q}, \dots, v_n^{p, q}\} \subset \R^{p, q}$,
\item There is a non-zero vector $w \in \R^{p, q}$ that is perpendicular to ${\rm Aff}\{v_1^{p, q}, \dots, v_n^{p, q}\}$,
\item $Z$ admits a magnitude weighting.
\end{enumerate}
\end{prop}
\begin{proof}
Since the projection $\pi^{p, q}$ is surjective, we have
\[
\R^{p, q} = {\rm Span}\{v_1^{p, q}, \dots, v_n^{p, q}\} = \R v_i + {\rm Span}\{v_2^{p, q}-v_1^{p, q}, \dots, v_n^{p, q}-v_1^{p, q}\},
\]
for all $i$. Note that $0\in{\rm Aff}\{v_1^{p, q}, \dots, v_n^{p, q}\}$ is equivalent to that $v_i \in {\rm Span}\{v_2^{p, q}-v_1^{p, q}, \dots, v_n^{p, q}-v_1^{p, q}\}$ for all $i$. Hence we obtain that
\begin{align*}
&0\in{\rm Aff}\{v_1^{p, q}, \dots, v_n^{p, q}\} \\
& \Leftrightarrow {\rm Span}\{v_2^{p, q}-v_1^{p, q}, \dots, v_n^{p, q}-v_1^{p, q}\} = \R^{p, q} \\
& \Leftrightarrow \dim {\rm Span}\{v_2^{p, q}-v_1^{p, q}, \dots, v_n^{p, q}-v_1^{p, q}\} = p+q.
\end{align*}
Note that we have $\dim {\rm Aff}\{v_1^{p, q}, \dots, v_n^{p, q}\} \geq p+q-1$ by Proposition \ref{affine}. If $\dim {\rm Aff}\{v_1^{p, q}, \dots, v_n^{p, q}\} = p+q-1$, there exists a non-zero vector $w  \in \R^{p, q}$ that is perpendicular to the affine subspace ${\rm Aff}\{v_1^{p, q}, \dots, v_n^{p, q}\}$ by Proposition \ref{choco}. Conversely, if there exists such a non-zero vector $w  \in \R^{p, q}$, then we have $\dim {\rm Aff}\{v_1^{p, q}, \dots, v_n^{p, q}\} \leq \dim {\rm ker}\langle w, -\rangle_{p, q} = p+q-1$. Also, the non-degeneracy of $\R^{p, q}$ implies that $w \neq 0$ is equivalent to $\langle w, v_i^{p, q}\rangle_{p, q} \neq 0$ for all $i$. Namely, we obtain
\begin{align*}
&\dim {\rm Aff}\{v_1^{p, q}, \dots, v_n^{p, q}\} = p+q-1 \\
&\Leftrightarrow \exists w \in \R^{p, q}, \langle w, {\rm Span}\{v_2^{p, q}-v_1^{p, q}, \dots, v_n^{p, q}-v_1^{p, q}\}\rangle_{p, q} = 0, w\neq 0 \\
&\Leftrightarrow \exists w \in \R^{p, q}, \langle w, {\rm Span}\{v_2^{p, q}-v_1^{p, q}, \dots, v_n^{p, q}-v_1^{p, q}\}\rangle_{p, q} = 0, \forall i, \langle w, v_i^{p, q}\rangle_{p, q} \neq 0.
\end{align*}

Now we have 
\begin{align*}
& 0 \not\in {\rm Aff}\{v_1^{p, q}, \dots, v_n^{p, q}\}\\
&\Leftrightarrow \dim {\rm Aff}\{v_1^{p, q}, \dots, v_n^{p, q}\} = p+q-1\\
&\Leftrightarrow \exists w \in \R^{p, q}, \langle w, {\rm Span}\{v_2^{p, q}-v_1^{p, q}, \dots, v_n^{p, q}-v_1^{p, q}\}\rangle_{p, q} = 0 \text{ and }\forall i, \langle w, v_i^{p, q}\rangle_{p, q} \neq 0 \\
&\Leftrightarrow \exists \tilde{w} \in \R^{n}, \langle \tilde{w}, {\rm Span}\{v_2-v_1, \dots, v_n-v_1\}\rangle_{p, q, r} = 0 \text{ and }\forall i, \langle \tilde{w}, v_i\rangle_{p, q, r} \neq 0 \\
&\Leftrightarrow \exists \tilde{w}\in \R^{n}, \forall i, \langle \tilde{w}, v_i-v_1\rangle_{p, q, r} = 0,\langle \tilde{w}, v_i\rangle_{p, q, r} \neq 0 \\
&\Leftrightarrow \exists \tilde{w}\in \R^{n}, \exists c\neq 0\in \R, V^tI_{p, q, r}\tilde{w} = c1_n \\
&\Leftrightarrow \exists w'\in \R^{n}, \exists c\neq 0\in \R, V^tI_{p, q, r}Vw' = Zw' = c1_n \\
&\Leftrightarrow \exists w''\in \R^{n},  Zw'' = 1_n \\
&\Leftrightarrow Z\text{ admits a magnitude weighting}.
\end{align*}
This completes the proof.
\end{proof}

\begin{prop}\label{mag0}
Suppose that $Z$ admits a weighting. Then the following are equivalent.
\begin{enumerate}
\item The affine subspace ${\rm Aff}\{v_1^{p, q}, \dots, v_n^{p, q}\}$ is non-degenerate,
\item Every non-zero vector perpendicular to ${\rm Aff}\{v_1^{p, q}, \dots, v_n^{p, q}\}$ is non-lightlike,
\item ${\rm Mag}\ Z \neq 0$.
\end{enumerate}
\end{prop}
\begin{proof}
Since we have $\dim {\rm Aff}\{v_1^{p, q}, \dots, v_n^{p, q}\} = p+q-1$ by Proposition \ref{weighting},  ${\rm Aff}\{v_1^{p, q}, \dots, v_n^{p, q}\}$ is non-degenerate if and only if there exists a non-lightlike vector $w \in \R^{p, q}$ that is perpendicular to the affine subspace ${\rm Aff}\{v_1^{p, q}, \dots, v_n^{p, q}\}$, by Proposition \ref{codim1}. Now we have
\begin{align*}
&\exists w  \in \R^{p, q}, \text{ non-lightlike and perpendicular to }{\rm Aff}\{v_1^{p, q}, \dots, v_n^{p, q}\}\\
&\Leftrightarrow \exists w  \in \R^{p, q}, \langle w, w\rangle_{p, q} \neq 0, \langle w, {\rm Span}\{v_2^{p, q}-v_1^{p, q}, \dots, v_n^{p, q}-v_1^{p, q}\}\rangle_{p, q} = 0 \\
&\Leftrightarrow  \exists \tilde{w}  \in \R^{n}, \langle \tilde{w}, \tilde{w}\rangle_{p, q, r} \neq 0, \langle  \tilde{w}, {\rm Span}\{v_2-v_1, \dots, v_n-v_1\}\rangle_{p, q, r} = 0 \\
&\Leftrightarrow \exists  \tilde{w}\in \R^{n}, \langle  \tilde{w},  \tilde{w}\rangle_{p, q, r} \neq 0, \forall i, \langle v_i-v_1,  \tilde{w}\rangle_{p, q, r} = 0 \\
&\Leftrightarrow \exists \tilde{w} \in \R^{n}, \exists c\in \R, V^tI_{p, q, r}\tilde{w} = c1_n, \langle \tilde{w}, \tilde{w}\rangle_{p, q, r} \neq 0 \\
&\Leftrightarrow \exists w'\in \R^{n}, \exists c\in \R, V^tI_{p, q, r}Vw' = Zw' = c1_n, \langle w', w'\rangle_{Z} \neq 0 \\
&\Leftrightarrow \exists w'\in \R^{n},  Zw' = \frac{\langle w', w'\rangle_{Z}}{\sum_iw'_i}1_n, \sum_iw'_i\neq 0, \langle w', w'\rangle_{Z} \neq 0 \\
&\Leftrightarrow \exists w''\in \R^{n},  Zw'' = 1_n, \sum_iw''_i\neq 0 \\
&\Leftrightarrow {\rm Mag}\ Z \neq 0.
\end{align*}
This completes the proof.
\end{proof}

\begin{df}\label{causal}
Let $Z \in M_n(\R)$ be a symmetric matrix with $Z_{ii} = 1$ for all $i$ that admits a magnitude weighting. Choosing linearly independent vectors  $v_1, \dots, v_n \in \R^n$ satisfying $Z = V^tI_{p, q, r}V$ with $V = (v_1\ \dots\ v_n)$, we call  $Z$ {\it spacelike, timelike} or {\it lightlike} if every non-zero vector perpendicular to ${\rm Aff}\{v_1^{p, q}, \dots, v_n^{p, q}\}$ is spacelike, timelike or lightlike respectively. 
\end{df}
\begin{rem}
Note that Proposition \ref{mag0} implies that the following are equivalent: 
\begin{enumerate}
\item $Z$ is lightlike,
\item The affine subspace ${\rm Aff}\{v_1^{p, q}, \dots, v_n^{p, q}\}$ is degenerate,
\item ${\rm Mag}\ Z = 0$.
\end{enumerate}
Also, if $Z$ is not lightlike, the non-zero vector perpendicular to ${\rm Aff}\{v_1^{p, q}, \dots, v_n^{p, q}\}$ is determined uniquely up to a scalar multiplication.
\end{rem}
\begin{rem}
The above causal classification does not depend on the choice of $V$ since the signature of ${\rm Aff}\{v_1^{p, q}, \dots, v_n^{p, q}\}$  depends only on the inner products of $v_i$'s, which are the components of $Z$. 
\end{rem}

\begin{proof}[Proof of Theorem \ref{main}]
By Proposition \ref{mag0}, the affine subspace $ {\rm Aff}\{v_1^{p, q}, \dots, v_n^{p, q}\}$ is non-degenerate and its dimension is $p+q-1$. Also, by Proposition \ref{section}, there exists $K\neq 0$ such that ${\rm Aff}\{v_1^{p, q}, \dots, v_n^{p, q}\}\cap S_{p, q}(1) \cong S_{p', q'}(K^{-1})$. Hence there exists a vector $x \in {\rm Aff}\{v_1^{p, q}, \dots, v_n^{p, q}\}$ satisfying that 
\[
{\rm Aff}\{v_1^{p, q}, \dots, v_n^{p, q}\}\cap S_{p, q}(1) = \{v \in {\rm Aff}\{v_1^{p, q}, \dots, v_n^{p, q}\} \mid \langle x-v, x-v\rangle_{p, q} = K^{-1}\}.
\]
In particular, we have $\langle x-v_i^{p, q}, x-v_i^{p, q}\rangle_{p, q} = \langle x-v_j^{p, q}, x-v_j^{p, q} \rangle_{p, q}  = K^{-1}$ for all $1\leq i, j \leq n$. Since we have $x \in {\rm Aff}\{v_1^{p, q}, \dots, v_n^{p, q}\}$, there exists a vector $a = (a_1, \dots, a_{n})^t$ such that $\sum_ia_i=1$ and $x = V^{p, q}a$. Then we have $x-v_i^{p, q} = V^{p, q}a-V^{p, q}e_i = V^{p, q}(a-e_i)$, and hence
\begin{align*}
&1\leq \forall i, j \leq n, \langle x-v_i^{p, q}, x-v_i^{p, q}\rangle_{p, q} = \langle x-v_j^{p, q}, x-v_j^{p, q} \rangle_{p, q}, \\
&\Leftrightarrow 1\leq \forall i, j \leq n, \langle V^{p, q}(a-e_i), V^{p, q}(a-e_i)\rangle_{p, q} = \langle V^{p, q}(a-e_j), V^{p, q}(a-e_j)\rangle_{p, q}, \\
&\Leftrightarrow 1\leq \forall i, j \leq n, \langle V(a-e_i), V(a-e_i)\rangle_{p, q, r} = \langle V(a-e_j), V(a-e_j)\rangle_{p, q, r}, \\
&\Leftrightarrow 1\leq \forall i, j \leq n, \langle a-e_i, a-e_i\rangle_{Z} = \langle a-e_j, a-e_j\rangle_{Z}, \\
&\Leftrightarrow 1\leq \forall i, j \leq n, \langle a, a\rangle_{Z}-2\langle e_i, a\rangle_{Z}+\langle e_i, e_i\rangle_{Z}  = \langle a, a\rangle_{Z}-2\langle e_j, a\rangle_{Z}+\langle e_j, e_j\rangle_{Z},  \\
&\Leftrightarrow 1\leq \forall i, j \leq n, \langle e_i, a\rangle_{Z} = \langle e_j, a\rangle_{Z}, \\
&\Leftrightarrow \exists c \in \R, Za = c\cdot 1_n.
\end{align*}
 Since $Za = c\cdot 1_n$ implies $\langle a, e_i\rangle_{Z} = \langle a, a\rangle_{Z} = c$, we obtain
 \[
K^{-1} = \langle x-v_i, x-v_i\rangle_{p, q} = \langle a-e_i, a-e_i\rangle_{Z} = \langle a, a\rangle_{Z} -2\langle a, e_i\rangle_{Z} + 1 = 1-c,
\]
which implies $c = 1-K^{-1}$. Note that we have $c \neq 0$ by Proposition \ref{section}.  Hence $Za = c\cdot 1_n$ implies that $a/c$ is a magnitude weighting of $Z$, and we obtain
\[
{\rm Mag}\ Z = c^{-1} =  \frac{1}{1 - K^{-1}}.
\]
\end{proof}

\begin{cor}
We have 
\[
\begin{cases} Z\text{ is spacelike } \Leftrightarrow {\rm Mag}\ Z > 0 \\
Z\text{ is timelike } \Leftrightarrow {\rm Mag}\ Z < 0 \\
Z\text{ is lightlike } \Leftrightarrow {\rm Mag}\ Z = 0.
\end{cases}
\]
\end{cor}
Suppose that $Z$ is non-degenerate. As stated in Proposition \ref{standard}, the inner product space $\R^{p, q}$ is isomorphic to $(\R^n, \langle -, -\rangle_{Z})$ by the linear map define by $x \mapsto V^{-1}x$. By this isomorphism, the quasi-sphere $S_{p, q}(1)$ is mapped to the quasi-sphere $S_Z(1) = \{x \in \R^n \mid  \langle x, x\rangle_{Z} = 1\}$ isometrically, its center $Va$ is mapped to $a$, and the vectors $v_1, \dots, v_n$ is mapped to $e_1, \dots, e_n$ respectively. Also, the affine subspace ${\rm Aff}\{v_1, \dots, v_n\}$ is mapped to $\{x \in \R^n \mid \sum_i x_i = 1\}$. Hence Theorem \ref{main} restricted to non-degenerate $Z$ is equivalent to the following (the general case can be stated in principle but is complicated).

\begin{thm}\label{main'}
Let $Z \in M_n(\R)$ be a non-degenerate symmetric matrix with $Z_{ii} = 1$ for all $i$. We consider  $(\R^n, \langle -, -\rangle_{Z})$ as a pseudo-Riemannian manifold, and  we  also consider  the submanifold $S_Z(1) = \{x \in \R^n \mid  \langle x, x\rangle_{Z} = 1\}$ of $\R^n$ as a pseudo-Riemannian submanifold of $(\R^n, \langle -, -\rangle_{Z})$. Let $A = \{x \in \R^n \mid \sum_i x_i = 1\}$. Then ${\rm Mag}\ Z \neq 0$ implies that
\[
{\rm Mag}\ Z = \frac{1}{1 - K^{-1}},
\]
where $K\neq 0$ is the sectional curvature of the quasi-sphere $A\cap S_{Z}(1)$. Furthremore,  for the center $a \in A$ of the quasi-sphere $A\cap S_{Z}(1)$, a magnitude weighting $w$ of $Z$ is described as
\[
w = ({\rm Mag}\ Z)a   =  \frac{1}{\langle a, a\rangle_{Z}}a.
\]
In particular, the following are equivalent : 
\begin{enumerate}
\item $Z$ has a positive weighting,
\item ${\rm Mag}\ Z>0$, and the center $a$ of $A\cap S_{Z}(1)$ belongs to the interior of ${\rm Conv}\{e_1, \dots, e_n\}$.
\end{enumerate}
\end{thm}
\section{Upper bound for magnitude}\label{sec:upbd}
In this section, we give a proof of Theorem \ref{upbd}.

\begin{prop}\label{tensor}
Let $Z \in M_n(\R)$ be a positive semi-definite symmetric matrix with $Z_{ii} = 1$ for all $i$ and $\rank \  Z = p$. Suppose that $Z^{(2)}:=(Z_{ij}^2)_{ij}$ is positive definite. Then we have
\[
{\rm Mag}\ Z^{(2)} \leq p.
\]
When $Z$ is positive definite, the equality holds if and only if $Z = I_n$.
\end{prop}
\begin{proof}
Let $\langle -, - \rangle_{F_p} $ be the modified Frobenius inner product on $M_n(\R)$ defined by  
\[
\langle A, B \rangle_{F_p} = \sum_{i, j= 1}^{p} A_{ij}B_{ij}
\]
 for $A, B \in M_n(\R)$. Note that this is a positive semi-definite inner product on $M_n(\R)$. We choose a non-degenerate matrix $V$ satisfying $Z = V^t I_{p, 0, n-p}V$. Then we define linear maps 
\begin{align*}
&\Delta : (\R^n, \langle -, - \rangle_{Z^{(2)}})  \to (\R^n\otimes \R^n, \langle -, - \rangle_{Z\otimes Z}), \\
&V\otimes V : (\R^n\otimes \R^n, \langle -, - \rangle_{Z\otimes Z}) \to (\R^n\otimes \R^n, \langle -, - \rangle_{I_{p, 0, n-p}\otimes I_{p, 0, n-p}}), \\
&\varphi : (\R^n\otimes \R^n, \langle -, - \rangle_{I_{p, 0, n-p}\otimes I_{p, 0, n-p}}) \to (M_n(\R), \langle -, - \rangle_{F_p}),
\end{align*}
by
\begin{align*}
&\Delta(e_i) = e_i\otimes e_i, \\
& V\otimes V(e_i\otimes e_j) = v_i \otimes v_j, \\
&\varphi(e_i \otimes e_j) = e_ie_j^t.
\end{align*}
Note that all of them preserve inner products. Also, we can easily verify that  $\Delta$ is injective, and that both $V\otimes V$ and $ \varphi$ are isomorphisms. Now let $S_{Z^{(2)}}(1) = \{v \in \R^n \mid \langle v, v\rangle_{Z^{(2)}} = 1\}$ be the unit sphere in the inner product space $(\R^n, \langle -, -\rangle_{Z^{(2)}})$. Then, for the center $a \in {\rm Aff}\{e_1, \dots, e_n\}$ of the sphere ${\rm Aff}\{e_1, \dots, e_n\}\cap S_{Z^{(2)}}(1)$, Theorem \ref{main'} implies that
\[
{\rm Mag}\ Z^{(2)} = \langle a, a\rangle_{Z^{(2)}}^{-1}.
\]
Setting $f : = \varphi \circ (V\otimes V) \circ \Delta$, we have 
\begin{align*}
 \langle a, a\rangle_{Z^{(2)}} &= \inf \{ \langle x, x\rangle_{Z^{(2)}} \mid x \in  {\rm Aff}\{e_1, \dots, e_n\}\} \\
&= \inf \{ \langle X, X\rangle_{F_p} \mid X \in  {\rm Aff}\{f(e_1), \dots, f(e_n)\}\} \\
&= \inf \{ \langle X, X\rangle_{F_p} \mid x \in  {\rm Aff}\{v_1v_1^t, \dots, v_nv_n^t\}\}.
\end{align*}
The assumption that  $Z_{ii} = 1$ for all $i$ implies $1 =  \langle I_n, v_iv_i^t\rangle_{F_p}$, hence we obtain that $\langle I_n, X\rangle_{F_p} = 1$ for all $X \in  {\rm Aff}\{v_1v_1^t, \dots, v_nv_n^t\} $. Then the Cauchy--Schwarz inequality implies
\[
1 = \langle I_n, X\rangle_{F_p} \leq  \langle I_n, I_n\rangle_{F_p} \langle X, X\rangle_{F_p} =  p\langle X, X\rangle_{F_p},
\]
and hence we obtain ${\rm Mag}\ Z^{(2)}\leq p$. The equality holds if and only if $\langle X - \lambda I_n, X-\lambda I_n\rangle_{F_p} = 0$ for some $\lambda \in \R$. When $Z$ is positive definite, namely $p = n$, we have 
\begin{align*}
\langle X - \lambda I_n, X-\lambda I_n\rangle_{F_p} = 0 & \Leftrightarrow X = \lambda I_n \\
& \Leftrightarrow  X = \frac{1}{n} I_n \\
& \Leftrightarrow \sum_i\lambda_iv_iv_i^t = \frac{1}{n}I_n \\
& \Leftrightarrow V{\rm diag}(\lambda_i)V^t = \frac{1}{n}I_n \\
& \Leftrightarrow {\rm diag}(\lambda_i)Z = \frac{1}{n} \\
& \Leftrightarrow Z = I_n.
\end{align*}
Thus the equality ${\rm Mag}\ Z^{(2)}= n$ holds only when $Z = I_n$ if $Z$ is positive definite.
\end{proof}

\begin{proof}[Proof of Theorem \ref{upbd}]
Since we have $Z_{X} = Z_{\frac{1}{2}X}^{(2)}$, the positive definiteness of $Z_{\frac{1}{2}X}$ and Theorem \ref{tensor} implies the claim. 
\end{proof}

\begin{rem}
In Proposition \ref{tensor}, if $\mathrm{rank}\ Z = p < n$, then the matrices $Z$ realizing the equality ${\rm Mag}\ Z^{(2)} = p$ are generally not unique. To illustrate this fact, we here describe a construction producing examples with $n = 3$ and $p = 2$ (which can be generalized to the case where $n = p + 1$). Let $v_1, v_2, v_3 \in \R^2$ be vectors of unit norm with respect to the standard inner product which are pairwise linearly independent. We define $Z \in M_3(\R)$ as their Gram matrix $Z = (\langle v_i, v_j \rangle )_{ij}$. By design, the symmetric matrix $Z$ is positive semi-definite, $Z_{ii} = 1$ and $\mathrm{rank}\  Z = 2$.  Furthermore, it turns out that $Z^{(2)}$ is positive definite and ${\rm Mag}\ Z^{(2)}=2$.

To see this result, we let $x_i, y_i$ be the entries of the vector $v_i = (x_i, y_i)^t$. For $i = 1, 2, 3$, define $u_i \in \R^3$ by $u_i = (x_i^2, \sqrt{2}x_iy_i, y_i^2)^t$, and put $U = (u_1\  u_2\  u_3) \in M_3(\R)$. It is easy to verify $U^t\cdot U = Z^{(2)}$ and
$
\det \ U
=
-\sqrt{2}
(x_1y_2 - x_2y_1)
(x_1y_3 - x_3y_1)
(x_2y_3 - x_3y_2)$.
Hence $Z^{(2)}$ is positive definite under our choice of $v_i$. It is also easy to verify $U^t\iota = 1_3$, where $\iota = (1,0,1)^t$. Now, it follows that
$
{\rm Mag}\ Z^{(2)} 
=
1_3^t(U^tU)^{-1}1_3
=
(U^t\iota)^t (U^tU)^{-1} (U^t\iota)
=
\iota^t\iota
= 2$.
\end{rem}

\begin{rem}
    The map $\Delta$ in the proof of Theorem \ref{upbd} is a well-known one so called {\it Veronese embedding}.
\end{rem}

\section{Geometric interpretations of the other facts on magnitude}\label{sec:geomin}
\subsection{Criterion for the existence of magnitude weighting and positive weighting}

Proposition \ref{weighting} gives a geometric criterion for the existence of a magnitude weighting. Also, the last parts of Theorems \ref{main} and \ref{main'} give a geometric criterion for the existence of a positive weighting. To make it easier to understand the situation intuitively, we restate them for the case that $Z$ is positive semi-definite.

\begin{prop}\label{psd}
Let $Z\in M_n(\R)$ be a positive semi-definite symmetric matrix with $Z_{ii} = 1$ for all $i$ and $\rank \  Z = p$. Choose a square matrix $V = (v_1\ \dots \ v_n)$ satisfying that $Z = V^tI_{p, 0, n-p} V$. 
\begin{enumerate}
\item $Z$ admits a magnitude weighting if and only if the affine subspace ${\rm Aff}\{v^{p, 0}_1, \dotsm v^{p, 0}_n\}$ does not contain the origin $0 \in \R^p$. Equivalently, $Z$ admits a magnitude weighting if and only if $ \dim {\rm Aff}\{v^{p, 0}_1, \dotsm v^{p, 0}_n\} = p-1$.
\item Let $c$ be the center of the $(p-2)$-dimensional sphere that appears as the intersection of ${\rm Aff}\{v^{p, 0}_1, \dotsm v^{p, 0}_n\}$ and the unit sphere in $\R^p$. Then $Z$ admits a positive weighting if and only if $c$ is in the interior of the convex hull ${\rm Conv}\{v^{p, 0}_1, \dotsm v^{p, 0}_n\}$.
\end{enumerate}
\end{prop}

We can obtain criteria for the existence of a magnitude weighting and a positive weighting for positive semi-definite matrices with small ranks as follows.

\begin{prop}\label{rank1}
Let $Z\in M_n(\R)$ be a positive semi-definite symmetric matrix with $Z_{ii} = 1$ for all $i$ and $\rank \  Z = 1$.  
If $Z \neq 1_n 1_n^t$, then $Z$ admits no weighting.  If $Z = 1_n 1_n^t$, then $Z$ admits a positive weighting.
\end{prop}

\begin{proof}
Let  $Z = V^tI_{1, 0, n-1} V$ with $V = (v_1\  \dots\  v_n)$.  Note that $\rank \  V  =1$ implies  $ \dim {\rm Aff}\{v^{1, 0}_1, \dots, v^{1, 0}_n\}  =1$ except when $v_i = v_j$ for all $i, j$, equivalently $Z = 1_n 1_n^t$. Hence Proposition \ref{psd} implies the claim. The latter statement is obvious.
\end{proof}

\begin{prop}\label{rank2}
Let $Z\in M_n(\R)$ be a positive semi-definite symmetric matrix with $Z_{ii} = 1$ for all $i$  and $\rank \  Z = 2$.  
Then $Z$ admits a magnitude weighting if and only if it is equivalent to a matrix of the form
\[
\begin{pmatrix}
1_{k, k} & c\,1_{k,n-k}\\
c\, 1_{n-k,k} & 1_{n-k, n-k}
\end{pmatrix},
\]
where $1_{p,q} := 1_p1_q^t$ and $c \in\R$. In this case, $Z$ admits a positive weighting.
\end{prop}

\begin{proof}
Let $Z = V^tI_{2, 0, n-2} V$ with $V = (v_1\  \dots\  v_n)$. By, Proposition \ref{psd}, $Z$ admits a magnitude weighting if and only if 
${\rm Aff}\{v^{2, 0}_1, \dots, v^{2, 0}_n\}$ is a line not passing through the origin.  
Since we have $v_i \in S^{1}$ for all $i$, ${\rm Aff}\{v^{2, 0}_1, \dots, v^{2, 0}_n\}$ is a line not passing through the origin if and only if the cardinality of the  set $\{v^{2, 0}_1, \dots, v^{2, 0}_n\}$ is $2$. Hence $Z$ admits a magnitude weighting if and only if  $Z = V^tI_{2, 0, n-2} V$ is equivalent to a matrix of the above form.
Moreover, ${\rm Conv}\{v^{2, 0}_1, \dots, v^{2, 0}_n\}$ is the line segment obtained as the intersection of 
${\rm Aff}\{v^{2, 0}_1, \dots, v^{2, 0}_m\}$ with the unit disk, whose midpoint coincides with the center of $S^0$ where $v_i$'s lie.  Hence  Proposition \ref{psd} implies that $Z$ admits a positive weighting.
\end{proof}

When $Z \in M_3(\R)$ is a positive definite symmetric matrix, we obtain the following necessary and sufficient condition for the existence of a positive weighting.

\begin{prop}\label{3by3}
Let $Z\in M_3(\R)$ be a positive definite symmetric matrix with $Z_{ii} = 1$ for all $i$. Then $Z$ admits a positive weighting if and only if 
\[
Z_{ij} + Z_{jk} - Z_{ik} < 1, 
\]
for $\{i,j,k\} = \{1,2,3\}$.
\end{prop}

\begin{proof}
Let $Z = V^t\cdot V$ with $V = (v_1\ v_2\  v_3)$.  
Then $v_1, v_2, v_3$ lie on a unit circle $\Sigma$, and by Proposition \ref{psd},  
$Z$ admits a positive weighting if and only if the interior of the triangle with vertices $v_1, v_2, v_3$ contains the center of $\Sigma$.  A triangle’s circumcenter lies inside the triangle if and only if it is acute, which is equivalent to
\[
|v_i - v_j|^2 + |v_j - v_k|^2 - |v_i - v_k|^2 > 0,
\]
for $\{i,j,k\} = \{1,2,3\}$. Since $|v_i - v_j|^2 = 2 - 2\langle v_i, v_j\rangle = 2 - 2Z_{ij}$, this condition becomes
\[
Z_{ij} + Z_{jk} - Z_{ik} < 1, 
\]
for $\{i,j,k\} = \{1,2,3\}$.
\end{proof}

It is difficult to obtain a condition similar to Proposition \ref{3by3} for higher degree matrices. However, the following  well-known fact combined with Proposition \ref{psd} shows that the chance of obtaining a positive weighting is very small.

\begin{prop}\label{putnum}\cite{HS}
When $n+1$ distinct points are randomly chosen from an $(n-1)$-dimensional sphere,  
the probability that their convex hull contains the center of the sphere is $2^{-n}$.
\end{prop}

Next, we give sufficient conditions for $Z$ not to admit a positive weighting.

\begin{prop}\label{pw1}
Let $Z\in M_n(\R)$ be a positive definite symmetric matrix with $Z_{ii} = 1$ for all $i$.
If there exists  $1 \leq i \leq n$ such that, for all $1 \leq j \leq n$,
\[
Z_{ij} \geq \frac{1}{{\rm Mag}\ Z},
\]
then $Z$ does not admit a positive weighting.
\end{prop}

\begin{proof}
Let $Z = V^t\cdot V$ with $V = (v_1\  \dots\  v_n)$. Let $\Sigma$ be the circumsphere of $\{v_1, \dots, v_n\}$ with the center $c$.  
By Proposition \ref{psd}, the following are equivalent: 
\begin{itemize}
\item  $Z$ does not admit a positive weighting,
\item  $c$ is not in the interior of ${\rm Conv}\{v_1, \dots, v_n\}$, 
\item  The vectors $v_i$'s are  in a common hemisphere of $\Sigma$.  
\end{itemize}
Let $d_\Sigma$ denote the spherical distance on $\Sigma$, and let $R$ be the radius of $\Sigma$.  
If there exists $1 \leq i \leq n$ such that for all $1 \leq j \leq n$,
\begin{align}\label{dsigma}
d_\Sigma(v_i, v_j) \leq \pi R / 2,
\end{align}
then $\{v_1, \dots, v_n\}$ are in the hemisphere centered at $v_i$.  Let $\theta_{ij}$ be the angle $\angle v_i c v_j $.  Then the inequality (\ref{dsigma}) is equivalent to  $\cos \theta_{ij} \geq 0$.  
Since
\[
2R^2 - 2R^2 \cos \theta_{ij} = |v_i - v_j|^2 = 2 - 2Z_{ij},
\]
we have $\cos \theta_{ij} = 1 - (1 - Z_{ij})/r^2$, and the inequality (\ref{dsigma}) is equivalent to
\[
\frac{1}{1 - Z_{ij}} \geq R^{-2} = \frac{{\rm Mag}\ Z}{{\rm Mag}\ Z - 1}.
\]
Solving this inequality gives
\[
Z_{ij} \ge \frac{1}{{\rm Mag}\ Z}.
\]
\end{proof}

\begin{prop}\label{pw2}
Let $Z\in M_n(\R)$ be a positive definite symmetric matrix with $Z_{ii} = 1$ for all $i$.
If we have
\[
\frac{2}{n}\sum_p Z_{ip} - 1 \ge \frac{1}{{\rm Mag}\ Z},
\]
for all $i$, then $Z$ does not admit a positive weighting.
\end{prop}

\begin{proof}
Let $Z = V^t\cdot V$ with $V = (v_1\  \dots\  v_n)$, and set $m := \frac{1}{n}V1_n$.  Let $w$ be the magnitude weighting of $Z$. Then Theorem \ref{main} implies that the center of the circumsphere of $\{v_1, \dots, v_n\}$ is $\overline{w} := Vw / {\rm Mag}\ Z$.  
It is clear that the inequality
\[
\max_i |m - v_i|^2 \le |m - \overline{w}|^2
\]
implies that $Z$ does not admit a positive weighting.  Since we have
\begin{align*}
|m - v_i|^2 
&= \langle m - v_i, m - v_i \rangle_{I_n} \\
&= \langle \frac{1}{n}V1_n - Ve_i, \frac{1}{n}V1_n - Ve_i \rangle_{I_n} \\
&= \langle \frac{1}{n}1_n - e_i, \frac{1}{n}1_n - e_i \rangle_Z \\
&= \frac{1}{n^2}\langle 1_n, 1_n \rangle_Z - \frac{2}{n}\langle 1_n, e_i \rangle_Z + \langle e_i, e_i \rangle_Z \\
&= \frac{1}{n^2}\sum_{ij} Z_{ij} - \frac{2}{n}\sum_p Z_{ip} + 1,
\end{align*}
 Proposition \ref{barycircum} implies that
\begin{align*}
&\max_i |m - v_i|^2 \le |m - \overline{w}|^2 \\
&\Leftrightarrow \frac{1}{n^2}\sum_{ij} Z_{ij} - \frac{2}{n}\min_i \sum_p Z_{ip} + 1 \le \frac{1}{n^2}\sum_{ij} Z_{ij} - \frac{1}{{\rm Mag}\ Z} \\
&\Leftrightarrow \frac{2}{n}\min_i \sum_p Z_{ip} - 1 \geq \frac{1}{{\rm Mag}\ Z} \\
&\Leftrightarrow \forall i,\ \frac{2}{n}\sum_p Z_{ip} - 1 \geq \frac{1}{{\rm Mag}\ Z}.
\end{align*}
This completes the proof.
\end{proof}

\begin{rem}
Proposition \ref{pw1} supports the following intuition for the magnitude of a finite metric space. We use Willerton’s {\it penguin valuation}  \cite{W1} that interprets the magnitude weighting as a thermal distribution in a group of penguins that tend to maintain the thermal balance :  
If there exists a penguin $i$ that is very close to all other penguins $j$,  
it receives too much heat, and to maintain thermal balance at each point,  
penguin $i$ must emit negative heat.  
Indeed, if $d(x_i, x_j)$ is very small, then $Z_{ij}$ is close to $1$,  
and since ${\rm Mag}\ Z > 1$ for a positive definite $Z$, it is likely that $Z_{ij} \geq 1 / {\rm Mag}\ Z$.  
Thus the weighting acquires negative components, corresponding to the emission of negative heat.  
Note, however, that this argument assumes $Z$ is positive definite.
\end{rem}

\subsection{Rayleigh quotient-like expression of magnitude}
For a positive semi-definite symmetric matrix $Z$ admitting a magnitude weighting, Leinster--Meckes \cite{LM} showed the following Rayleigh quotient-like formula : 
\[
{\rm Mag}\ Z = \sup_{\substack{ a \in \R^n \\ a^tZa \neq 0}} \frac{(\sum_ia_i)^2}{a^tZa}.
\]
Furthermore,  if $Z$ is positive definite,  then the supremum is attained by exactly the nonzero
scalar multiples of the magnitude weighting  of $Z$. We can understand this formula geometrically as follows. Let $\rank \  Z = p$ and $Z = V^tI_{p, 0, n-p} V$ with $V = (v_1\ \dots \ v_n)$. Then we have 
\begin{align*}
{\rm Aff}\{v_1^{p, 0}, \dots, v_n^{p, 0}\} &= \{V^{p, 0}a \mid a\in \R^n, \sum_ia_i = 1\} \\
&= \{V^{p, 0}a/(\sum_ia_i) \mid a\in \R^n, \sum_ia_i \neq 0\}.
\end{align*}
 Note that the center of the sphere ${\rm Aff}\{v_1^{p, 0}, \dots, v_n^{p, 0}\}  \cap S_{p, 0}(1)$ is the orthogonal projection $P$ of the origin $0 \in \R^p$ onto ${\rm Aff}\{v_1^{p, 0}, \dots, v_n^{p, 0}\} $. We also note that, since $Z$ admits a magnitude weighting, Theorem \ref{psd} implies  $\langle V^{p, 0}a/(\sum_ia_i), V^{p, 0}a/(\sum_ia_i) \rangle_{I_p} \neq 0$ if $\sum_ia_i \neq 0$. Hence  $\sum_ia_i \neq 0$ implies $\langle a, a \rangle_{Z} = a^tZa \neq 0$. Since the norm of the orthogonal projection $|P|^2$ is the infimum of the norm of vectors on ${\rm Aff}\{v_1^{p, 0}, \dots, v_n^{p, 0}\}$,  we have
\begin{align*}
 \langle P, P\rangle_{I_p}&= \inf_{\substack{a\in\R^n \\ \sum_ia_i \neq 0}}\langle V^{p, 0}a/(\sum_ia_i), V^{p, 0}a/(\sum_ia_i) \rangle_{I_p} \\
 &= \inf_{\substack{a\in\R^n \\ \sum_ia_i \neq 0}}\langle Va/(\sum_ia_i), Va/(\sum_ia_i) \rangle_{I_n} \\
&= \inf_{\substack{a\in\R^n \\ \sum_ia_i \neq 0}}\frac{\langle a, a \rangle_{Z}}{(\sum_ia_i)^2}\\
&= \inf_{\substack{a\in\R^n \\ \sum_ia_i \neq 0, a^tZa\neq 0}}\frac{a^tZa}{(\sum_ia_i)^2}.
\end{align*}
Now Theorem \ref{thm1} implies that 
\begin{align*}
{\rm Mag}\ Z &= \langle P, P\rangle_{I_p}^{-1} \\
&= \sup_{\substack{a\in\R^n \\ \sum_ia_i \neq 0, a^tZa\neq 0}}\frac{(\sum_ia_i)^2}{a^tZa} \\
&= \sup_{\substack{a\in\R^n \\ a^tZa\neq 0}}\frac{(\sum_ia_i)^2}{a^tZa}.
\end{align*}
The supremum is attained when $V^{p, 0}a$ is a scalar multiple of $P$, namely $a$ is a scalar multiple of a magnitude weighting.
\subsection{Spread and magnitude}
\begin{prop}\label{barycircum}
Let $Z \in M_n(\R)$ be a symmetric matrix with $Z_{ii} = 1$ for all $i$. Suppose that $Z$ admits a magnitude weighting $w$ and ${\rm Mag}\ Z \neq 0$. Choose a decomposition  $Z = V^tI_{p, q, r} V$ with $V=(v_1\ \dots\ v_n)$. Let $m := \frac{1}{n}\sum_iv_i^{p, q}$ be the barycenter of the vectors $v_1^{p, q}, \dots, v_n^{p, q}$. Note that  the center of the quasi-sphere ${\rm Aff}\{v_1^{p, q}, \dots, v_n^{p, q}\} \cap S_{p, q}(1)$ is $\overline{w} := V^{p, q}w/{\rm Mag}\ Z$. Then we have
\[
|m-\overline{w}|_{I_{p, q}}^2 = \frac{\sum_{ij}Z_{ij}}{n^2}-\frac{1}{{\rm Mag}\ Z}.
\]
In particular, when $Z$ is semi-positive definite,  we have ${\rm Mag}\ Z\geq n^2/\sum_{ij}Z_{ij}$, and the following are equivalent for positive definite $Z$.
\begin{enumerate}
\item ${\rm Mag}\ Z= n^2/\sum_{ij}Z_{ij}$,
\item $m = \overline{w}$,
\item Each row of $Z$ has the same sum.
\item $1_n$ is an eigenvector of $Z$.
\end{enumerate}
\end{prop}
\begin{proof}
Since we have $m =  \frac{1}{n}\sum_iv^{p, q}_i =  \frac{1}{n}V^{p, q}1_{n}$, 
\begin{align*}
|m-\overline{w}|_{I_{p, q}}^2  &= \langle  \frac{1}{n}V^{p, q}1_n- \overline{w}, \frac{1}{n}V^{p, q}1_{n}- \overline{w}\rangle_{I_{p, q}} \\
&= \langle V^{p, q}( \frac{1}{n}1_{p+q}- \frac{1}{{\rm Mag}\ Z}w), V^{p, q}( \frac{1}{n}1_{n}- \frac{1}{{\rm Mag}\ Z}w)\rangle_{I_{p, q}} \\
&= \langle V( \frac{1}{n}1_{n}- \frac{1}{{\rm Mag}\ Z}w), V( \frac{1}{n}1_{n}- \frac{1}{{\rm Mag}\ Z}w)\rangle_{I_{p, q, r}} \\
&= \langle \frac{1}{n}1_{n}- \frac{1}{{\rm Mag}\ Z}w, \frac{1}{n}1_{n}- \frac{1}{{\rm Mag}\ Z}w\rangle_{Z} \\
&= \langle \frac{1}{n}1_n, \frac{1}{n}1_n\rangle_{Z} -2\langle  \frac{1}{{\rm Mag}\ Z}w, \frac{1}{n}1_n\rangle_{Z} +  \langle \frac{1}{{\rm Mag}\ Z}w, \frac{1}{{\rm Mag}\ Z}w\rangle_{Z} \\
&= \frac{1}{n^2}\sum_{ij}Z_{ij}-\frac{2}{{\rm Mag}\ Z} + \frac{1}{{\rm Mag}\ Z} \\
&= \frac{1}{n^2}\sum_{ij}Z_{ij}-\frac{1}{{\rm Mag}\ Z}.
\end{align*}
When $Z$ is positive definite, we also have
\begin{align*}
{\rm Mag}\ Z =  n^2/\sum_{ij}Z_{ij} &\Leftrightarrow |m-\overline{w}|_{I_{n}}^2= 0 \\
&\Leftrightarrow m=\overline{w}\\
&\Leftrightarrow w=\frac{{\rm Mag}\ Z}{n}1_n\\
&\Leftrightarrow \forall i, j, \sum_pZ_{ip} = \sum_pZ_{jp}.
\end{align*}
This completes the proof.
\end{proof}

The quantity $n^2/\sum_{ij}Z_{ij} $ is exactly the $2$-{\it spread} defined by Willerton \cite{W}, and the above inequality is also obtained by himself. Proposition \ref{barycircum} shows that the differnce between the $2$-spread and the magnitude is exactly the distance between the barycenter and the circumcenter of $v^{p, q}_i$'s.  He also defined the $Q$-{\it spread} $E_Q(X)$ for general $Q \in [0, \infty]$ and a finite metric space $X$ by
\[
E_Q(X) = \begin{cases} \left( \sum_i\frac{1}{n} (\frac{1}{n}\sum_p Z_{ip})^{Q-1} \right)^{\frac{1}{1-Q}} & Q \neq 1, \infty, \\
\Pi_i (\frac{1}{n}\sum_pZ_{ip})^{-\frac{1}{n}} & Q = 1, \\
\min_{i} \frac{1}{\frac{1}{n}\sum_pZ_{ip}} & Q = \infty,
\end{cases}
\]
where $Z = Z_X$. Since we have 
\[
\langle m, v_i^{p, q} \rangle_{I_{p, q}} = \frac{1}{n}\langle 1_{n}, e_i\rangle_Z = \frac{1}{n}\sum_pZ_{ip},
\]
we can rewrite the definition of the $Q$-spread as
\[
E_Q(X) = \begin{cases} \left( \sum_i\frac{1}{n} (\langle m, v_i^{p, q} \rangle_{I_{p, q}} )^{Q-1} \right)^{\frac{1}{1-Q}} & Q \neq 1, \infty, \\
\Pi_i (\langle m, v_i^{p, q} \rangle_{I_{p, q}} )^{-\frac{1}{n}} & Q = 1, \\
\min_{i} \frac{1}{\langle m, v_i^{p, q} \rangle_{I_{p, q}} } & Q = \infty.
\end{cases}
\]
Namely, they are {\it the power mean} of $\langle m, v_i^{p, q} \rangle_{I_{p, q}}$'s, which shows a non-uniformity of vectors $v_i^{p, q}$'s. In particular, we have $E_2(X) = |m|_{I_{p, q}}^2$, nameyl the $2$-spread is the norm of the barycenter.

\section{Another geometric description}\label{sec:another}

This section provides another geometric description of the magnitude.

\begin{thm}\label{thm:another}
Let $Z \in M_n(\R)$ be a symmetric matrix with $Z_{ii} = 1$ for all $i$. Suppose that the signature of $Z$ is $(p, q, r)$, and let $v_1, \dots, v_n \in \R^n$ be linearly independent vectors such that $Z = V^tI_{p, q, r}V$ with $V = (v_1\ \dots\ v_n)$. Suppose also that $Z$ admits a magnitude weighting $w \in \R^n$. Then we have
\[
{\rm Mag}\ Z = 4 R^2,
\]
where $R^2 \in \R$ is the radial scalar square of a quasi-sphere in $(\R^n, \langle -, - \rangle_{p,q,r})$ which circumscribes the points $\{0, v_1, \dots, v_n \}$. Furthermore, the center $c \in \R^n$ of this quasi-sphere is given by $c = \frac{1}{2}Vw$, and is unique if $Z$ is non-degenerate. 
\end{thm}

\begin{proof}
We denote $\langle v, v\rangle_{p, q, r}$ by $|v|^2_{p, q, r}$ for $v \in \R^n$ in the following.
A quasi-sphere which circumscribes $\{ 0, v_1, \dots, v_n \}$ is characterized by the equations
\begin{gather*}
| c |^2_{p, q, r} 
= | c - v_i |_{p, q, r}^2
\end{gather*}
for $i$, where $c \in \R^n$ is the center and the radial scalar square is $R^2 = | c |^2_{p,q,r}$. Under the assumption $| v_i |^2_{p,q,r} = Z_{ii} = 1$, the equations are equivalent to $\langle v_i, c \rangle_{p, q, r} = \frac{1}{2}$ for $i$, which is summarized in a single equation $V^tI_{p, q, r}c = \frac{1}{2}1_n$. Using the magnitude weighting, we see
\[
V^tI_{p,q,r}\bigg(\frac{1}{2}Vw \bigg)
=
\frac{1}{2} Zw
= 
\frac{1}{2} 1_n.
\]
Hence a center is given by $c = \frac{1}{2} Vw$. This is unique if $Z$ is non-degenerate, since the magnitude weighting is unique in this case. The radial scalar square is 
\[
R^2
=
| c |_{p, q, r}^2
=
\frac{1}{4}
(Vw)^tI_{p, q, r}(Vw)
=
\frac{1}{4}
w^tZw
=
\frac{1}{4}
w^t1_n
=
\frac{1}{4}{\rm Mag}\ Z,
\]
and the proof is completed.
\end{proof}

Specializing to the positive definite case, one has Theorem \ref{thm:another_description} in Section \ref{sec:introduction}:

\begin{cor}
Let $Z \in M_n(\R)$ be a positive definite symmetric matrix with $Z_{ii} = 1$ for all $i$.  Let $v_1, \dots, v_n \in \R^n$ be linearly independent vectors such that $Z = V^t\cdot V$ with $V = (v_1\ \dots\ v_n)$. Then we have
\[
{\rm Mag}\ Z = 4 R^2,
\]
where $R$ is the radius of the circumsphere of the points $\{0, v_1, \dots, v_n \}$. 
\end{cor}

\bigskip

It is potetially possible to reproduce the results in Section \ref{sec:upbd}--\ref{sec:geomin} by using the formula in Theorem \ref{thm:another}. We can indeed prove Proposition \ref{tensor} (which leads to Theorem \ref{upbd}) in Section \ref{sec:upbd} based on the geometric interpretation of the magnitude in Theorem \ref{thm:another}. An interesting thing is that a simplification of this geometric proof yields yet another proof of Proposition \ref{tensor}, which is purely linear algebraic:

\begin{prop}[Proposition \ref{tensor}] 
Let $Z \in M_n(\R)$ be a positive semi-definite symmetric matrix with $Z_{ii} = 1$ for all $i$ and $\rank \  Z = p$. Suppose that $Z^{(2)}:=(Z_{ij}^2)_{ij}$ is positive definite. Then we have
\[
{\rm Mag}\ Z^{(2)} \leq p.
\]
When $Z$ is positive definite, the equality holds if and only if $Z = I_n$.
\end{prop}

\begin{proof}
Take linearly independent vectors $v_1, \dots, v_n \in \R^n$ such that $V^tI_{p,0,n-p}V = Z$ with $V = (v_1\  \dots \ v_n)$. Let $\langle - , - \rangle_{F_p}$ be the modified Frobenius inner product on $M_n(\R)$ defined in the proof of Proposition \ref{tensor}, which is positive semi-definite. It follows that the Gram matrix $\mathrm{Gram}(I_{p,0,n-p},v_iv_i^t)$ of the vectors $I_{p,0,n-p},v_1v_1^t, \ldots, v_nv_n^t$ with respect to $\langle - , - \rangle_{F_p}$ is also positive semi-definite. A direct calculation leads to
\begin{align*}
\langle I_{p, 0, n-p}, I_{n,0,n-p} \rangle_{F_p}
&=
p,
&
\langle I_{p, 0, n-p}, v_iv_i^t \rangle_{F_p}
&=
1,
&
\langle v_iv_i^t, v_jv_j^t \rangle_{F_p}
&=
Z_{ij}^2
\end{align*}
for $i, j = 1, \ldots, n$. Therefore we have
\[
\mathrm{Gram}(I_{p,0,n-p},v_iv_i^t)
=
\left(
\begin{array}{cc}
p & 1_n^t \\
1_n & Z^{(2)}
\end{array}
\right).
\]
Since this is positive semi-definite, its determinant is non-negative. Recalling that $Z^{(2)}$ is assumed to be positive definite, we compute a Schur complement to get
\[
0 \le
\lvert 
\mathrm{Gram}(I_{p,0,n-p},v_iv_i^t)
\rvert
=
\lvert Z^{(2)} \rvert
(p - 1_n^t(Z^{(2)})^{-1}1_n)
=
\lvert Z^{(2)} \rvert
(p - {\rm Mag}\ Z^{(2)}).
\]
Therefore ${\rm Mag}\ Z^{(2)} \le p$. Suppose here that $p = n$ and $Z$ is positive definite. Then $\langle -, - \rangle_{F_p}$ is the usual Frobenius inner product, and hence is positive definite. Because $v_1, \dots, v_n$ are linearly independent, so are $v_1v_1^t, \dots, v_nv_n^t$. By the nature of the Gram matrix with respect to a non-degenerate inner product,  $\lvert 
\mathrm{Gram}(I_n,v_iv_i^t) \rvert = 0$ if and only if $I_n = \sum_{i = 1}^n \lambda_i v_iv_i^t = V \mathrm{diag}(\lambda_i)V^t$ for some $\lambda_i \in \R$. Now, by the same argument as in the proof of Proposition \ref{tensor}, we conclude that $Z = I_n$.
\end{proof}

\end{document}